\documentclass[12pt]{amsart}
\usepackage{latexsym,amsfonts,amsmath,amssymb,amsthm,url,amsbsy,amscd,mathrsfs}
\usepackage[english]{babel}
\usepackage[latin1]{inputenc}
\usepackage{graphicx,psfrag,epsfig}
\usepackage{enumerate}
\numberwithin{equation}{section}
\newtheorem{theorem}{Theorem}[section]
\newtheorem{proposition}[theorem]{Proposition}
\newtheorem{corollary}[theorem]{Corollary}
\newtheorem{lemma}[theorem]{Lemma}

\usepackage{color}
\definecolor{darkred}{rgb}{0.8,0,0}
\definecolor{darkblue}{rgb}{0,0,0.7}
\definecolor{darkgreen}{rgb}{0,0.4,0}

\newcommand{\C}{\mathbb{C}}

\newcommand{\N}{\mathbb{N}}

\newcommand{\R}{\mathbb{R}}

\newcommand{\Z}{\mathbb{Z}}
\renewcommand{\AA}{\mathscr{A}}

\newcommand{\FF}{\mathscr{F}}

\newcommand{\HH}{\mathscr{H}}

\newcommand{\LL}{\mathscr{L}}

\newcommand{\PP}{\mathscr{P}}

\renewcommand{\SS}{\mathscr{S}}

\newcommand{\VV}{\mathscr{V}}

\newcommand{\eps}{\varepsilon}

\renewcommand{\d}{{\rm d}}

\newcommand{\un}{{\rm 1\kern -2.5pt l}}

\newcommand{\de}{\partial}
\newcommand{\supp}{{\rm supp}}
\renewcommand{\div}{\mathrm{div}}

\usepackage[normalem]{ulem}     

\newcommand{\RRR}{\color{black}}

\begin{document}
\title[]{Fractional high order thin film equation: gradient flow approach} 
\date{\today}
\keywords{Fractional operator, thin film equation, Minimizing movements scheme, Gradient flow, Wasserstein distance}
\subjclass[2010]{35A01, 35R11, 35G25, 35K46, 49K20, 35B09}
\begin{abstract}
We prove existence of weak solutions of a fractional thin film type equation in any space dimension and for any order of the equation.
The proof is based on a gradient flow technique in the space of Borel probability measures endowed with the Wasserstein distance.

\end{abstract}

\author{S. Lisini}
\address{Stefano Lisini -- Universit\`a di Pavia, Dipartimento di Matematica "F. Casorati", via Ferrata 5, 27100 Pavia, Italy}
\email{stefano.lisini@unipv.it}

%
%
%

\thanks{}

\maketitle
\section{Introduction}


In this paper we prove existence of non-negative solutions of the following evolution  problem:
\begin{equation}\label{equation}
\begin{cases}
\partial_t u-\mathrm{div}(u\nabla (\LL_s u) )=0 &\mbox{in }(0,+\infty)\times \R^d , \\
u(0,\cdot)=u_0 &\mbox{in }{\R}^d,
\end{cases}
\end{equation}
where the operator $\LL_s$ is the $s$-fractional Laplacian on ${\R}^d$ and $s\in (0,+\infty)$. The order of the equation in \eqref{equation} is $2+2s$.
We assume that the initial datum $u_0\in L^1(\R^d)$ satisfies $u_0\geq 0$ and $ \displaystyle \int_{\R^d}|x|^2u_0(x)\,\d x <+\infty$. 

The linear operator $\LL_s$, also denoted by $(-\Delta)^s$, can be
defined using the Fourier transform by
\begin{equation}\label{defLs}
\widehat{\LL_su}(\xi):=|\xi|^{2s}\hat u(\xi).
\end{equation}
For the Fourier transform we use the following convention $\hat v(\xi):= \displaystyle \int_{{\R}^d}e^{-ix\cdot\xi}\, v(x)\,\d x$ for $v\in L^1(\R^d)$. 
Recalling the link between the Fourier transform and the differentiation, it is immediate to check that
for $s=1$, $\LL_1=-\Delta$ is the classical Laplacian, and for $s=2$, $\LL_2=(-\Delta)^2$ is the classical bi-Laplacian.
Also in the case $s\in \N$, we still use the terminology ``fractional Laplacian'' for the operator $\LL_s$. 

The equation in \eqref{equation} is a fractional version of a Thin Film type Equation.
The case $s\in(0,1)$ was recently studied in \cite{SV20} for existence of weak solutions and their asymptotic behavior.
The case $s=1$ in dimension $d=1$ 
has been treated in \cite{O98} and \cite{GO01}, where the gradient flow formulation 
has been highlighted. 
Still in the case $s=1$ the gradient flow formulation has been exploited in \cite{MMS09} to prove existence in any dimension.
For higher order of the equation, precisely for $s>1$, general existence theorems (at the knowledge of the author) 
are not available in the literature except in particular cases: 
in dimension $d=1$ and $s\in\N$ the first existence result was proved in \cite{BF90} 
(see also \cite{CG10} and \cite{FK04} for the case $d=1$ and $s=4$). Another interesting results in one dimension is contained in \cite{IM11}.

The aim of this paper is to prove an existence result for problem \eqref{equation}, for all order $s\in(0,+\infty)$,
using the gradient flow interpretation in the space of Probability measures endowed with the Wasserstein distance.
The problem of the proof of existence of weak solutions of \eqref{equation} using a gradient flow technique has been raised in Section 7 of \cite{SV20}.

We notice that this technique automatically gives the conservation of the mass and the non-negativity of the solutions.
From now on we assume that the initial mass $\int_{\R^d}u_0(x)\,\d x=1$.

\medskip
In the rest of the introduction we describe the technique and we illustrate the main result of the paper.

 \subsection*{The gradient flow setting and the main result}

We denote by $\PP_2(\R^d)$ the space of Borel probability measures on $\R^d$ 
with finite second moment. The space $\PP_2(\R^d)$, endowed with the 2-Wasserstein distance $W$,
 is a complete and separable metric space
(see Subsection \ref{Sub:Wasserstein} for the definition of $W$ and its properties).
For  $u\in \PP_2(\R^d)$ we define the energy functional
\begin{equation*}
	\FF_s(u):=\frac12\|u\|^2_{\dot H^{s}(\R^d)}
\end{equation*}
where $\|u\|_{\dot H^{s}(\R^d)}$ is the
seminorm of the homogeneous
Sobolev space ${\dot H^{s}({\R}^d)}$ (see Subsection \ref{SobolevHomSpace}) defined as follows
\begin{equation*}
	\|u\|^2_{\dot H^{s}(\R^d)}
	:=\frac1{(2\pi)^d} \int_{{\R}^d}|\xi|^{2s}|\hat u (\xi)|^2\,\d\xi.
\end{equation*}

%
We prove that
a solution of the Cauchy problem  \eqref{equation} can be obtained using the so-called 
minimizing movement approximation scheme (in the terminology introduced by De Giorgi \cite{DG93}),
applied to the functional $\FF_s$ in the metric space $(\PP_2(\R^d),W)$.
A general theory of  minimizing movements in metric spaces and its applications to the  space $(\PP_2(\R^d),W)$ 
is contained in the book of Ambrosio-Gigli-Savar\'e \cite{AGS}. 
The gradient flow approach in $(\PP_2(\R^d),W)$ with this approximation scheme
was first highlighted by Jordan-Kinderlehrer-Otto in the seminal paper \cite{JKO}
for the Fokker-Planck equation.
The first study of a fourth order equation using this technique was carried out by Gianazza-Savar\'e-Toscani \cite{GST}.
The gradient flow approach to a problem involving fractional laplacian operators is given in \cite{LMS18}.

Let us illustrate the strategy in our case: given $u_0\in\dot H^{s}(\mathbb{R}^d)\cap\PP_2(\mathbb{R}^d)$
we introduce the following time discretization scheme: we consider 
a time step $\tau>0$,
we set $u_\tau^0:=u_0$ and we recursively define
\begin{equation}\label{minmov1}
u_\tau^k\in{\rm Argmin}_{u\in\PP_2(\mathbb{R}^d)}\left\{ \FF_s(u)+\frac1{2\tau}\,W^2(u,u^{k-1}_\tau)\right\}, \qquad \mbox{for } k=1,2,\ldots.
\end{equation}
The existence and uniqueness of solution for the minimization problem in \eqref{minmov1} will be established in Subsection \ref{subsection:MM}. 
If $\{u_\tau^k\}_{k\in\mathbb{N}}\subset\PP_2(\mathbb{R}^d)$ is a sequence defined by \eqref{minmov1}, 
we introduce the piecewise constant interpolation
\begin{equation}\label{PC1}
	u_\tau(t):=u_\tau^{k}, \qquad \mbox{if } t\in ((k-1)\tau,k\tau], \quad k=1,2,\ldots, \qquad u_\tau(0):=u_\tau^{0}=u_0,
\end{equation}
We refer to $u_\tau$ as discrete solution.
The family of piecewise constant curves $\{u_\tau: \tau>0\}$ admits a limit curve, in a suitable sense, as $\tau\to 0$
and a limit curve is a weak solution of the equation in \eqref{equation}.

 We state the results in the following Theorem \ref{th:main}.

Before state the Theorem we point out that the space $AC^2([0,+\infty);(\PP_2(\mathbb{R}^d),W))$
is defined in Section \ref{subsection:MM}.
Moreover, we denote by $[a]:=\max\{n\in\Z:n\leq a\}$, the integer part of the real number $a$.

\begin{theorem}\label{th:main}
Let {\RRR $d\ge 1$, $s>0$}  and $u_0\in \dot H^{s}(\mathbb{R}^d)\cap\PP_2(\mathbb{R}^d)$.
Then the following assertions hold:
\begin{itemize}
\item[i)] {\bf Existence and uniqueness of discrete solutions.} For any $\tau>0$,  there exists a unique sequence
$\{u_\tau^k:k=0,1,2,\ldots\}$ satisfying \eqref{minmov1}. In particular the discrete solution 
$u_\tau:[0,+\infty)\to\PP_2(\R^d)$ in \eqref{PC1} is uniquely defined.
\item[ii)] {\bf Convergence and regularity.} For any vanishing sequence $\tau_n$ 
there exists a (non relabeled) subsequence  $\tau_{n}$ and a curve
$u\in AC^2([0,+\infty);(\PP_2(\mathbb{R}^d),W))$ such that:
\begin{enumerate}
\item $\forall\,r\in[0,s)$,  $u\in C([0,+\infty);H^r(\R^d))$ and
\[
 u_{\tau_{n}}(t) \to u(t) \quad \mbox{strongly in $H^r(\R^d)$ as $n\to\infty$, $\forall\,t\in[0,+\infty)$,}
\]
\item
$u\in C([0,+\infty);H^s_w(\R^d))$, where $H^s_w(\R^d)$ denotes the space $H^s(\R^d)$ endowed with the weak topology, and
\[
 u_{\tau_{n}}(t) \to u(t) \quad \mbox{weakly in $H^s(\R^d)$ as $n\to\infty$, $\forall\,t\in[0,+\infty)$,}
\]
\item
 $u\in L^2((0,T); H^{1+s}(\R^d))$ for every $T>0$, and
\[
 u_{\tau_{n}}\to u \quad \mbox{strongly in } L^2((0,T);H^{1+r}({\R}^d)) \quad \mbox{ as $n\to\infty$, $\forall\,r\in[0,s)$,}
\]
\[
 u_{\tau_{n}}\to u \quad \mbox{weakly in } L^2((0,T);H^{1+s}({\R}^d)) \quad \mbox{ as }n\to\infty.
\]
\end{enumerate}
\item[iii)] {\bf Solution of the equation.} $u$ satisfies the equation in \eqref{equation} in the following {\RRR weak form:}
\begin{equation}\label{wequation}
\int_0^{+\infty}\int_{\R^d} u \,\partial_t\varphi  \,\d x\,\d t +
\int_0^{+\infty}N(u(t,\cdot),\nabla\varphi(t,\cdot)) \, \d t=0 , \quad  \forall\, \varphi\in C^\infty_c((0,+\infty)\times\R^d),
\end{equation}
where $N:H^{1+s}(\R^d)\times  C^\infty_c(\R^d;\R^d)\to \R$ is defined by
\begin{equation}\label{N11}
	N(v,\eta):=\begin{cases}
	\displaystyle \int_{\R^d} ((\LL_{s-m} v) \LL_m(\div (\eta \, v))\,\d x, & \text{if } s\in[2m,2m+1], \\
	\displaystyle \int_{\R^d} \nabla((\LL_{s-m-1} v)\cdot\nabla(\LL_m(\div (\eta \, v)))\,\d x, & \text{if } s\in(2m+1,2m+2),
	\end{cases}
\end{equation}
and $m:=[s/2]$ is the integer part of $s/2$.  
\end{itemize}
\end{theorem}

We provide some comments about the statement of the Theorem.

The existence of the discrete solutions and a limit curve follows by a standard argument in the general theory of minimizing movements.
This fact is illustrated in Section \ref{Sec:functional}. 

The minimizers in \eqref{minmov1}, and consequently $u_\tau(t)$ for $t\geq0$, belong to $H^s(\R^d)$ by construction. 
This regularity and a bound on the $H^s$ norms allow to obtain the convergence in points  $ii) (1) (2)$.
The regularity $H^s$ of $u(t)$ is not sufficient to give a good notion of weak solution of the equation in  \eqref{equation}.
If the limit curve $u$ belong to $L^2((0,T); H^{1+s}(\R^d))$ we can give a good notion of weak solution of the equation
in \eqref{equation}, precisely \eqref{wequation}. Indeed the nonlinear operator $N$ in \eqref{N11} is well defined.
In order to show that minimizers has the further regularity $H^{s+1}$, we use variations along the Heat flow using the flow interchange technique
stated in \cite{MMS09}.
This further regularity allows to make variations along the flow generated by a smooth vector field,
that yield a discrete weak formulation of the equation.
A suitable bound in the space $L^2((0,T); H^{1+s}(\R^d))$ allows to obtain the convergence in  point  $ii) (3)$
 and to pass to the limit in the weak discrete formulation of the equation.

Finally we observe that the initial datum $u_0$ of the problem \eqref{equation} is attained in the sense of the continuity 
at the points $ii) (1) (2)$.

\medskip
We conclude the introduction  pointing out that for Thin film models it is interesting also a system of the type
\begin{equation}\label{equationMob}
\begin{cases}
\partial_t u-\mathrm{div}(m(u)\nabla (\LL_s u) )=0 &\mbox{in }(0,+\infty)\times \R^d , \\
u(0,\cdot)=u_0 &\mbox{in }{\R}^d,
\end{cases}
\end{equation}
where $m:[0,+\infty)\to[0,+\infty)$ is the so-called mobility function. 
In the case of concave non linear mobilities $m$ the  system \eqref{equationMob} 
is formally a gradient flow in the space of Borel probability measures endowed with a mobility dependent distance 
introduced in \cite{DNS09} (see also \cite{LM10}).
This technique was used in \cite{LMS12} for the Cahn-Hilliard equation.

\section{Notation and preliminary results} 
\subsection{Probability measures and Wasserstein distance}\label{Sub:Wasserstein}
For a detailed treatment of this topic see \cite{AGS} and \cite{V03}.
We denote by $\PP(\mathbb{R}^d)$ the set of Borel probability measures on $\mathbb{R}^d$.
The narrow convergence in $\PP(\mathbb{R}^d)$ is defined in duality with continuous and bounded functions on $\mathbb{R}^d$.
Precisely,
a sequence $\{u_n\}_{n\in\N}\subset \PP(\R^d)$ narrowly converges to $u\in  \PP(\R^d)$ if
$\int_{\R^d} \phi \, \d u_n \to \int_{\R^d} \phi \, \d u$ as $n\to+\infty$ for every $\phi\in C_b(\R^d)$,
where  $C_b(\R^d)$ is the set of continuous and bounded real functions on $\R^d$.

Given $u\in\PP(\R^k)$ and $G:\R^k\to \R^n$ a Borel measurable map, 
we define the push forward (or image measure) of $u$ through $G$, denoted by $G_\# u\in\PP(\R^n)$, by 
$G_\# u(B):=u(G^{-1}(B))$ for all Borel set
$B\subset \R^n$, or equivalently,
\[\int_{\R^n} f(y)\,\d G_\#u(y)=\int_{\R^k} f(G(x))\,\d u(x),\]
for every Borel positive function $f:\R^k\to\R$. 

Since in this paper we use only measures $u\in\PP(\R^d)$
 absolutely continuous with respect to the Lebesgue measure, we identify the measure $u$ with its density, 
 and with abuse of notation we write $\d u(x)=u(x)\d x$.
 
 We also recall that, when $u\in\PP(\R^d)$ is absolutely continuous with respect to the Lebesgue measure
 and $G:\R^d\to\R^d$ is a diffeomorphism, then $v:=G_\#u$ 
 is absolutely continuous with respect to the Lebesgue measure and
 \begin{equation}\label{densityPF}
  v(x)= u(G^{-1}(x))\det(\nabla G^{-1}(x)).
 \end{equation}

 We define $\PP_2(\R^d):=\{u\in\PP(\R^d):  \int_{\R^d} |x|^2 \, \d u(x)<+\infty \}$
the set of Borel probability measures with finite second moment.
The Wasserstein distance {\RRR W} in $\PP_2(\R^d)$ is defined as
\begin{equation}\label{Kanto}
W(u,v):=\min_{\gamma\in\PP({\R^d\times\R^d})}\left\{\left(\int_{\R^d\times\R^d}\!\!\!\!\!|x-y|^2\,\d\gamma(x,y)\right)^{1/2}:
\,(\pi_1)_\#\gamma=u,\,(\pi_2)_\#\gamma=v
\right\}
\end{equation}
where $\pi_i$, $i=1,2$, denote the canonical projections on the first and second factor respectively.

The function $W:\PP_2(\R^d)\times\PP_2(\R^d)\to[0,+\infty)$ is a distance and the metric space $(\PP_2(\R^d),W)$ is complete and separable.
The distance $W$ is sequentially lower semi continuous with respect to the narrow convergence, i.e.,
\begin{equation}\label{Wlsc}
 u_n\to u,\quad v_n\to v, \mbox{ narrowly } \Longrightarrow \liminf_{n\to+\infty}W(u_n,v_n)\geq W(u,v),
\end{equation}
and 
\begin{equation}\label{Wcomp}
\mbox{bounded sets in $(\PP_2(\R^d),W)$ are narrowly sequentially relatively compact.}
\end{equation}

\subsection{Fourier transform and fractional Sobolev spaces}\label{SobolevHomSpace}
We denote by $\SS(\R^d)$ the Schwartz space of smooth functions with rapid decay at infinity and by $\SS'(\R^d)$ the dual space of tempered distributions.
The Fourier transform of $u\in \SS(\R^d)$ is defined by 
$ \hat u(\xi):= \int_{\mathbb{R}^d}e^{-ix\cdot\xi}u(x)\,\d x. $
The Fourier transform is an automorphism of $\SS(\R^d)$ and can be defined on $\SS'(\R^d)$ by transposition.
Moreover the  Plancherel formula holds
\begin{equation}\label{Plancherel}
 \int_{\R^d} \hat u(\xi)\overline{\hat w(\xi)}\,\d\xi =  (2\pi)^d\int_{\R^d} u(x) w(x)\,\d x,\qquad \forall u,w\in L^2(\R^d).
 \end{equation}
We observe that if $u$ is real valued, then
\begin{equation}\label{Fconj}
	\overline{\hat u(\xi)}=\hat u(-\xi) \qquad \forall\,\xi\in\R^d.
\end{equation}
Moreover we recall the link between the Fourier transform and the differentiation, valid for tempered distributions,
\begin{equation}\label{Frule}
\widehat{\de_{x_k}u}(\xi)=i\xi_k\hat u(\xi), \qquad u\in\SS'(\R^d).
\end{equation}

Let $r\in\R$.
For every tempered distribution $u\in\SS'(\R^d)$ such that $\hat u \in L^1_{loc}(\R^d)$,
we define
\begin{equation*}
\|u \|^2_{H^{r}(\R^d)} := \frac1{(2\pi)^d}\int_{\mathbb{R}^d}(1+|\xi|^2)^{r}|\hat u (\xi)|^2\,\d\xi
\end{equation*}
and
\[
\|u \|^2_{\dot H^{r}(\R^d)} := \frac1{(2\pi)^d}\int_{\mathbb{R}^d}|\xi|^{2r}|\hat u (\xi)|^2\,\d\xi.
\]
By \eqref{Plancherel} and \eqref{Frule} it holds
$$ 
\|u \|^2_{H^{1}(\R^d)} = \int_{\mathbb{R}^d}(|u (x)|^2+|\nabla u(x)|^2)\,\d x, 
\qquad \|u \|^2_{\dot H^{1}(\R^d)} = \int_{\mathbb{R}^d}|\nabla u(x)|^2\,\d x.
$$
The fractional Sobolev space $H^{r}(\R^d)$ is defined by
$$H^{r}(\R^d):=\{u\in\SS'(\R^d): \hat u \in L^1_{loc}(\R^d) ,\; \|u \|_{H^{r}(\R^d)}<+\infty \},$$
and the homogenous fractional Sobolev space  $\dot H^{r}(\R^d)$ is defined by 
$$\dot H^{r}(\R^d):=\{u\in\SS'(\R^d): \hat u \in L^1_{loc}(\R^d) ,\; \|u \|_{\dot H^{r}(\R^d)}<+\infty \}.$$
Using Plancherel's formula \eqref{Plancherel} and the definition of $\LL_s$ in \eqref{defLs} 
it is immediate to show that 
\begin{equation}
	 \|u\|^2_{\dot H^{r}(\R^d)}=\|\LL_{r/2}u\|^2_{L^{2}(\R^d)}, \qquad \forall\, r>0.
\end{equation}
Moreover, it follows from the definition of $\LL_s$ that
\begin{equation}
	 u\in{\dot H^{r}(\R^d)} \quad\Longrightarrow \quad \LL_{s}u\in{\dot H^{r-2s}(\R^d)}.
\end{equation}
In this paper we use  the following obvious relations: 
$$ \|u \|_{H^{r_1}(\R^d)} \leq \|u \|_{H^{r_2}(\R^d)},\qquad \mbox{if } r_1<r_2,$$
$$\|u \|_{\dot H^{r}(\R^d)} \leq \|u \|_{H^{r}(\R^d)}, \qquad \mbox{if }r>0,$$
$$\|u \|_{\dot H^{0}(\R^d)} = \|u \|_{H^{0}(\R^d)}= \|u \|_{L^{2}(\R^d)},$$
and the interpolation inequalities
\begin{equation}\label{interpolation}
\begin{split}
&\|u \|_{H^{r_1}(\R^d)}\leq \|u \|^{1-\theta}_{H^{r_0}(\R^d)} \|u \|^\theta_{H^{r_2}(\R^d)}\quad\text{and}
\quad \|u \|_{\dot H^{r_1}(\R^d)}\leq \|u \|^{1-\theta}_{\dot H^{r_0}(\R^d)} \|u \|^\theta_{\dot H^{r_2}(\R^d)},\\
 &\mbox{if } r_0<r_1<r_2 \mbox{ and $\theta$ satisfies } r_1=(1-\theta)r_0+\theta r_2,
\end{split}
\end{equation}
(see for instance \cite[Sections 1.3, 1.4]{BCD}).

%
%

The following lemma and its Corollary will be useful in the sequel for proving convergence results.

\begin{lemma}\label{lemma:conv}
Let $s>0$ and $\{u_n\}_{n\in\N}\subset H^{s}({\R}^d)$ a sequence such that
$\sup_{n\in\N}\|u_n\|_{H^s(\R^d)}<+\infty$
and 
\begin{equation}\label{eq:BFT}
	\sup_{n\in\N}\sup_{\xi\in B_R(0)}|\hat u_n(\xi)| <+\infty, \qquad \forall\,R>0.
\end{equation}
If $u$ is a Borel signed measure such that
 $\hat u_n(\xi) \to \hat u(\xi)$ for every $\xi\in\R^d$, as $n\to+\infty$,
 then $u\in H^s(\R^d)$,
 $\|u_n-u\|_{H^r(\R^d)}\to 0$ as $n\to+\infty$,
  for any $r\in[0,s)$
  and $u_n\to u$ weakly in $H^s(\R^d)$ as $n\to+\infty$.

Moreover, if $h\in [0,s/2)$, then
$\|\LL_hu_n-\LL_hu\|_{H^r(\R^d)}\to 0$ for any $r\in[0,s-2h)$
and $\LL_hu_n\to \LL_hu$ weakly in $H^{s-2h}(\R^d)$.
Finally,  $\LL_{s/2}u_n\to \LL_{s/2}u$ weakly in $L^{2}(\R^d)$.
\end{lemma}
\begin{proof}
We first prove that $u\in H^s(\R^d)$ and $u_n\to u$ weakly in $H^s(\R^d)$.\\
We define $U_n(\xi):=(1+|\xi|^2)^{s/2}\hat u_n(\xi)$.
By assumption we have
 $\sup_n\|U_n\|_{L^2(\R^d)}<+\infty$
and
$U_n(\xi)\to (1+|\xi|^2)^{s/2}\hat u (\xi)$ for every $\xi\in\R^d$ as $n\to+\infty$.
By the $L^2$ weak compactness there exists a subsequence of $\{U_n\}$ weakly convergent in $L^2(\R^d)$ 
to some $U\in L^2(\R^d)$.
By uniqueness of the weak limit we have that 
$U(\xi)=(1+|\xi|^2)^{s/2}\hat u(\xi)$.
By the lower semicontinuity of the $L^2$ norm we obtain that
 $\|u\|_{H^s(\R^d)} \le\liminf_{n\to\infty}\|u_n\|_{H^s(\R^d)}$.
 Since the weak topology is metrizable in bounded sets and
 the limit is independent of the subsequence, all the sequence $u_n$ converges weakly in $H^s(\R^d)$.

Let us fix $r\in [0,s)$ and we prove that $u_n$ strongly converges to $u$ in $H^r(\R^d)$.\\
 For any $R>0$ we write
\begin{equation}\label{ciccia0}
\begin{split}
	\|u_n-u\|^2_{ H^r(\R^d)}&=\int_{\R^d}(1+|\xi|^2)^{r}|\hat u_n(\xi)-\hat u(\xi)|^2\,\d\xi \\
	&=\int_{\{|\xi|\leq R\}}(1+|\xi|^2)^{r}|\hat u_n(\xi)-\hat u(\xi)|^2\,\d\xi \\
	&+ \int_{\{|\xi|>R\}}(1+|\xi|^2)^{r}|\hat u_n(\xi)-\hat u(\xi)|^2\,\d\xi .
\end{split}\end{equation}
Let $C$ be a constant such that $\|u_n\|^2_{H^s(\R^d)}\leq C$. Since, as observed in the first part of the proof, also 
$\|u\|^2_{H^s(\R^d)}\leq C$, we have
\begin{equation}\begin{split}\label{ciccia1}
		\int_{\{|\xi|>R\}}(1+|\xi|^2)^{r}|\hat u_n(\xi)-\hat u(\xi)|^2\,\d\xi &= \int_{\{|\xi|>R\}}(1+|\xi|^2)^{s}(1+|\xi|^2)^{(r-s)}|\hat u_n(\xi)-\hat u(\xi)|^2\,\d\xi \\
		&\leq (1+R^2)^{(r-s)}\int_{\{|\xi|>R\}}(1+|\xi|^2)^{s}|\hat u_n(\xi)-\hat u(\xi)|^2\,\d\xi\\
		 &\leq (1+R^2)^{(r-s)}\|u_n-u\|^2_{H^s(\R^d)} 
		 \leq 4C (1+R^2)^{(r-s)}.
	\end{split}\end{equation}
On the other hand, since $(1+|\xi|^2)^{r}|\hat u_n(\xi)-\hat u(\xi)|^2\to 0$ for any $\xi\in\R^d$ and by \eqref{eq:BFT}
 $(1+|\xi|^2)^{r}|\hat u_n(\xi)-\hat u(\xi)|^2\leq(1+R^2)^{r}C_R^2$, where $C_R:=\sup_{n\in\N}\sup_{\xi\in B_R(0)}|\hat u_n(\xi)|$, by dominated convergence Theorem we have
\begin{equation}\label{ciccia2}
 \lim_{n\to+\infty}\int_{\{|\xi|\leq R\}}(1+|\xi|^2)^{r}|\hat u_n(\xi)-\hat u(\xi)|^2\,\d\xi =0.
 \end{equation}
Fixing $\eps>0$, by \eqref{ciccia1} there exists $R$ such that  $\int_{\{|\xi|>R\}}(1+|\xi|^2)^{r}|\hat u_n(\xi)-\hat u(\xi)|^2\,\d\xi <\eps$
for any $n\in\N$. Then, by \eqref{ciccia0}, \eqref{ciccia1} and \eqref{ciccia2} we obtain that
 $\limsup_{n\to+\infty} \|u_n-u\|^2_{ H^r(\R^d)}\leq\eps$.
For the arbitrariness of $\eps$ we conclude.

The last assertions are consequence of the proved convergences and
the following relation 
$$\|\LL_hu_n\|_{\dot H^r(\R^d)}=\|u_n\|_{\dot H^{r+2h}(\R^d)} \leq \|u_n\|_{H^{r+2h}(\R^d)}.$$

\end{proof}

\begin{corollary}\label{cor:conv}
Let $s>0$ and
$\{u_n\}_{n\in\N}\subset \dot H^{s}({\R}^d)\cap\PP({\R}^d)$ a sequence 
such that $\sup_{n\in\N}\|u_n\|_{\dot H^s(\R^d)}<+\infty$.
If $u\in\PP(\R^d)$ and $u_n$ narrowly converges to $u$,
then $u_n, u\in H^s(\R^d)$,
and all the conclusions of Lemma \ref{lemma:conv} hold.
%
\end{corollary}
\begin{proof}
We prove that the assumptions of Lemma \ref{lemma:conv} hold.
Since $u_n$ is a density of a probability measure, then $|\hat u_n(\xi)|\leq 1$ for any $\xi\in\R^d$.
Then \eqref{eq:BFT} holds.
In order to prove that $\sup_{n\in\N}\|u_n\|_{H^s(\R^d)}<+\infty$,
by definition of the $H^s(\R^d)$ norm,
it is sufficient to prove that $\sup_{n\in\N}\|u_n\|_{L^2(\R^d)}<+\infty$.
Denoting by $B_1$ the unitary ball in $\R^d$, we have
\[ \int_{\R^d}|\hat u_n(\xi) |^2\,\d\xi = \int_{B_1}|\hat u_n(\xi) |^2\,\d\xi+\int_{\R^d\setminus B_1}|\hat u_n(\xi) |^2\,\d\xi
\leq |B_1| + \int_{\R^d}|\xi|^{2s}|\hat u_n(\xi) |^2\,\d\xi,\]
and by Plancherel formula we obtain
\begin{equation}
	(2\pi)^{d}\| u_n\|^2_{L^2(\R^d)}\leq |B_1|+ \| u_n\|^2_{\dot H^s(\R^d)}.
\end{equation}
Moreover the narrow convergence of $u_n$ to $u$ implies the pointwise convergence
$\hat u_n(\xi) \to \hat u(\xi)$ for any $\xi\in\R^d$.
\end{proof}

\section{Energy functional and first convergence result}\label{Sec:functional}

\subsection{Energy functional}
After noticing that a Borel probability measure $u$ is a tempered distribution with $\hat u$ in $L^1_{loc}(\R^d)$, we
define the energy functional
$\FF_s:\PP_2(\R^d)\to [0,+\infty]$
by
\begin{equation*}\label{functional}
\FF_s(u):=\frac12\|u\|^2_{\dot H^{s}(\mathbb{R}^d)}.
\end{equation*}
We denote by  $D(\FF_s)=\{u\in\PP_2(\R^d):\FF_s(u)<+\infty\}$.
Using Corollary \ref{cor:conv}, it is immediate to prove the following Proposition.

\begin{proposition}\label{prop:FP}
The following assertions hold:
\begin{itemize}
\item $D(\FF_s) = H^{s}(\R^d)\cap\PP_2(\R^d)$.
\item $\FF_s$ is sequentially lower semicontinuous w.r.t. the narrow convergence.
\end{itemize}
\end{proposition}

\subsection{Wasserstein gradient flow, minimizing movements}\label{subsection:MM}

We consider, for $k=1,2,\ldots$, the  problem
\begin{equation}\label{minmov}
\min_{u\in\PP_2(\mathbb{R}^d)} \FF_s(u)+\frac1{2\tau}\,W^2(u,u^{k-1}_\tau).
\end{equation}
\begin{proposition}\label{prop:existenceMM}
For every $\tau>0$ and every $u_0\in D(\FF_s)$ there exists a unique sequence
$\{u_\tau^k:k=0,1,2,\ldots\}\subset D(\FF_s)$
satisfying $u_\tau^0=u_0$ and such that $u_\tau^k$ is a solution to problem \eqref{minmov} for $k=1,2,\ldots$.
\end{proposition}
\begin{proof}
Let $\tau>0$ and $k\in\N$.
By Proposition \ref{prop:FP} and the properties of the Wasserstein distance \eqref{Wlsc} \eqref{Wcomp},
the functional $u\mapsto  \FF_s(u)+\frac1{2\tau}\,W^2(u,u^{k-1}_\tau)$ is nonnegative,
lower semicontinuous with respect to the narrow convergence and with narrowly compact sublevels.
The existence of minimizers  follows by standard direct methods in calculus of variations.
The uniqueness of minimizers follows from the strict convexity of the functional $ u\mapsto \FF_s(u)+\frac1{2\tau}\,W^2(u,u^{k-1}_\tau)$
with respect to linear convex combinations in $\PP_2(\R^d)$.
\end{proof}

By Proposition \ref{prop:existenceMM},  the piecewise constant curve
\begin{equation}\label{PC}
	u_\tau(t):=u_\tau^{k}, \qquad \mbox{if } t\in ((k-1)\tau,k\tau], \quad k=1,2,\ldots, \qquad u_\tau(0):=u_\tau^{0}=u_0,
\end{equation}
is uniquely defined.
%

We say that a curve $u:[0,+\infty)\to \PP_2(\R^d)$ is absolutely continuous with finite energy, and we use the notation
$u\in AC^2([0,+\infty);(\PP_2(\R^d),W))$, if there exists
$m\in L^2([0,+\infty))$ such that $W(u(t_1),u(t_2))\leq \int_{t_1}^{t_2} m(r)\,dr$ for any $t_1,t_2\in[0,+\infty)$, $t_1<t_2$.

\begin{theorem}[First convergence result]\label{th:convergence1}
Let $u_0\in D(\FF_s)$ and $u_\tau$ the piecewise constant curve defined in \eqref{PC}.
For every vanishing sequence $\tau_n$ there exists a subsequence (not relabeled) $\tau_{n}$ and a curve
$u\in AC^2([0,+\infty);(\PP_2(\mathbb{R}^d),W))$ such that
\begin{equation}\label{narrowconv}
u_{\tau_{n}}(t) \to u(t)\quad \mbox{ narrowly as $n\to\infty$, for any $t\in[0,+\infty)$}.
\end{equation}
\end{theorem}
\begin{proof}

The first estimate given by the scheme \eqref{minmov}, is the following
\begin{equation}\label{basicestimate}
\FF_s(u_\tau^N) + \frac12\sum_{k=1}^{N} \tau\frac{W^2(u_\tau^k,u_\tau^{k-1})}{\tau^2} \leq \FF_s(u_\tau^0) =  \FF_s(u_0), 
\qquad \forall\,N\in\N.
\end{equation}

We show that for any $T>0$ the set $\AA_T:=\{u_\tau^N:\tau>0, N\in\N, N\tau\leq T\}$ is bounded in $(\PP_2(\R^d), W)$ and by \eqref{Wcomp} sequentially narrowly compact.\\
Indeed, recalling that $\int_{\R^d}|x|^2\,\d u(x)= W^2(u,\delta_0)$ for any $u\in\PP_2(\R^d)$,
using the triangle inequality for $W$ and Jensen's discrete inequality we have
\begin{equation}\label{boundW}
\begin{aligned}
   \int_{\R^d}|x|^2u_\tau^N(x)\,\d x= W^2(u_\tau^N,\delta_0) & \leq \Big(\sum_{k=1}^{N} W(u_\tau^k,u_\tau^{k-1}) + W(u_\tau^0,\delta_0)\Big)^2 \\
    &\leq 2\Big(\sum_{k=1}^{N} \tau\frac{W(u_\tau^k,u_\tau^{k-1})}{\tau}\Big)^2 +2 W^2(u_\tau^0,\delta_0) \\
    &\leq 2N\tau \sum_{k=1}^{N} \tau\frac{W^2(u_\tau^k,u_\tau^{k-1})}{\tau^2} +2 W^2(u_\tau^0,\delta_0).
\end{aligned}
\end{equation}
Since $\FF_s\geq 0$,   from \eqref{basicestimate} and \eqref{boundW} it follows that 
\begin{equation}\label{boundA}
	\int_{\R^d}|x|^2u_\tau^N(x)\,\d x \leq 2T\FF_s(u_0) + 2\int_{\R^d}|x|^2u_0(x)\,\d x, \qquad \forall \,N\in\N: \;N\tau\leq T
\end{equation}
and the boundedness of $\AA_T$ follows.

We define the piecewise constant function $m_\tau:[0,+\infty)\to [0,+\infty)$ as
$$ m_\tau(t):=\frac{W(u_\tau(t),u_\tau(t-\tau))}{\tau}$$
with the convention that $u_\tau(t-\tau)=u_\tau(0)$ if $t-\tau<0$.
Since $\FF_s\geq 0$, from \eqref{basicestimate} it follows that
\begin{equation*}\label{basicestimate2}
 \frac12\int_{0}^{+\infty} m_\tau^2(t)\,\d t \leq  \FF_s(u_0).
\end{equation*}
It follows that there exists $m\in L^2(0,+\infty)$ such that $m_\tau$ weakly converges to $m$ in $L^2(0,+\infty)$.
Moreover for any $t_1,t_2 \in [0,+\infty)$,  $t_1<t_2$, setting $k_1(\tau)=[t_1/\tau]$ and $k_2(\tau)=[t_2/\tau]$, 
by triangle inequality it holds
\begin{equation*}\label{equi1}
\begin{aligned}
   W(u_\tau(t_1),u_\tau(t_2)) & \leq \sum_{k=k_1(\tau)}^{k_2(\tau)-1} W(u_\tau^k,u_\tau^{k-1}) \leq \int_{k_1(\tau)\tau}^{k_2(\tau)\tau}m_\tau(t)\,\d t .
   \end{aligned}
\end{equation*}
By the $L^2$ weak convergence of $m_\tau$ the following equicontinuity estimate holds
\begin{equation}\label{equicont}
	\limsup_{\tau\to 0} W(u_\tau(t_1),u_\tau(t_2))\leq \lim_{\tau\to 0} \int_{k_1(\tau)\tau}^{k_2(\tau)\tau}m_\tau(t)\,\d t = \int_{t_1}^{t_2}m(t)\,\d t.
\end{equation}
Applying Proposition 3.3.1 of \cite{AGS} we obtain the convergence  \eqref{narrowconv}.
Passing to the limit in \eqref{equicont} we obtain
\begin{equation*}\label{ac2}
	 W(u(t_1),u(t_2))\leq \int_{t_1}^{t_2}m(t)\,\d t, \qquad \forall\, t_1,t_2 \in [0,+\infty),\quad t_1<t_2,
\end{equation*}
and then $u\in AC^2([0,+\infty);(\PP_2(\mathbb{R}^d),W))$.
\end{proof}

\section{Estimates on discrete solutions, convergence and weak solution}\label{Sec:flow} 
In this Section
we briefly review the ``flow interchange estimate'' introduced by Matthes-McCann-Savar\'e in \cite{MMS09}.
Using this estimate with the entropy functional, we  obtain a suitable regularity estimate for the
family of discrete solutions $u_\tau$.
Moreover, using this estimate with a family of suitable potential energy functionals, 
we obtain that $u_\tau$ satisfies an approximate weak formulation of the equation in \eqref{equation}.

%

\subsection{Flow interchange technique}
We say that a lower semi continuous functional $\VV:\PP_2(\R^d)\to(-\infty,+\infty]$, 
with proper domain $D(\VV)=\{u\in\PP_2(\R^d): \VV(u)<+\infty\}\not=\emptyset$, 
generates a $\lambda$-flow, for $\lambda\in\R$, if there exists a
continuous semigroup $S_t:D(\VV)\to D(\VV)$ 
such that
the following family of \emph{Evolution Variational Inequalities} 
\begin{equation}\label{EVI}
      \limsup_{t\to 0^+} \frac{W^2(S_t(u),v)-W^2(u,v)}{2t} +  \frac{\lambda}{2}W^2(u,v)  \le \VV(v)-\VV(u), \quad \forall u\in D(\VV),
\end{equation}
hold.
 We recall that a continuous semigroup is a family of maps $S_t:D(\VV)\to D(\VV)$, $t\geq 0$, such that
 $$ S_t(S_r(u))=S_{t+r}(u), \quad \forall\, t,r\geq 0,\qquad \lim_{t\to 0^+}W(S_t(u),u)=0, \quad \forall u\in D(\VV). $$

%
%
%

If $u\in D(\FF_s)$ we define the dissipation of $\FF_s$ along the flow $S_t$ of $\VV$ by
\begin{equation}\label{def:DVF}
 \mathfrak{D}_\VV \FF_s(u):= \limsup_{t \to 0^+} \frac{\FF_s(u)-\FF_s(S_t(u))}{t}.
\end{equation}

\begin{proposition}[\textbf{Flow interchange}]\label{prop:FI}
Let $\{u_\tau^k:k=0,1,2,\ldots\}$ be the sequence given by {\rm Proposition \ref{prop:existenceMM}}, $\lambda\in\R$ and $\VV$  a 
functional generating a $\lambda$-flow.
If $u_\tau^k\in D(\VV)$
then 
\begin{equation}\label{FIE}
 \mathfrak{D}_\VV \FF_s(u_{\tau}^k) + \frac{\lambda}{2\tau}W^2(u_\tau^k,u_\tau^{k-1}) \le   \frac{\VV(u_{\tau}^{k-1}) -\VV(u_{\tau}^k)}{\tau}, \qquad \, k= 1,2,\ldots.
\end{equation}
\end{proposition}

\begin{proof}
For  $t>0$ and $k\geq 1$,
by definition of minimizer there holds
\[
    \FF_s(u_\tau^k) + \frac{1}{2\tau}W^2(u_\tau^k,u_\tau^{k-1})\leq \FF_s(S_t(u_\tau^k)) + \frac{1}{2\tau}W^2(S_t(u_\tau^k),u_\tau^{k-1}),
\]
that is,
\[
   \tau\frac{\FF_s(u_\tau^k) -\FF_s(S_t(u_\tau^k))}{t} \leq \frac{W^2(S_t(u_\tau^k),u_\tau^{k-1}) - W^2(u_\tau^k,u_\tau^{k-1})}{2t}.
\]
By using \eqref{EVI} and the definition \eqref{def:DVF} we obtain \eqref{FIE}.
\end{proof}

The next two propositions summarize well known results (see \cite{AGS} Theorems 11.2.5 and 11.2.3).
\begin{proposition}\label{prop:EntropyFlow}
The entropy functional $\HH:\PP_2(\R^d)\to (-\infty,+\infty]$ 
 defined by 
 \begin{equation*}
 \HH(u):=\begin{cases}
 	\displaystyle\int_{\R^d}u\log u \,\d x &\mbox{if $u$ is absolutely continuous w.r.t. Lebesgue measure,}\\
 	+\infty &\mbox{otherwise}, 
	\end{cases}
\end{equation*}
generates a $0$-flow.
The semigroup associated $u_t:=S_t(\bar u)$ is the unique solution of the Cauchy problem for the heat equation
$$\begin{cases}
\de_t u_t=\Delta u_t, & \mbox{in } (0,+\infty)\times\R^d\\
 \quad u_0=\bar u & \mbox{in } \R^d.
 \end{cases}$$
\end{proposition}

\begin{proposition}\label{prop:PotentialFlow}
Let $\varphi\in C^\infty_c(\R^d)$ and $\lambda \geq \|D^2\varphi\|_\infty$.
The functional $\VV:\PP_2(\R^d)\to \R$ defined by 
$$\VV(u):=\int_{\R^d}\varphi(x)\,\d u(x)$$ 
generates a $(-\lambda)$-flow $S_t$ and $S_t(u)=(X_t)_\#u$, where $x\mapsto X_t(x)$, $x\in\R^d$, 
is the map defined by the system
\begin{equation}\label{flowRdG}
\begin{cases}
\frac{\d}{\d t} X_t(x)=-\nabla\varphi(X_t(x)), & t\in\R\\
X_0(x)=x.
\end{cases}
\end{equation}
\end{proposition}
We observe that under the assumptions of Proposition \ref{prop:PotentialFlow}, $X_t$ is defined also for $t<0$.

\subsection{Improved regularity and dissipation along the Heat flow}
The following result makes use of flow interchange
with the choice $\VV=\HH$, the entropy functional.

%
%

\begin{lemma} \label{lemma:ed}
Let $u_0\in D(\FF_s)$ and $\{u_\tau^k:k=0,1,2,\ldots\}$ the sequence given by  {\rm Proposition \ref{prop:existenceMM}}.
Then
$u_\tau^k\in H^{1+s}(\R^d)$ for any $k\ge 1$ and
\begin{equation}\label{regestimate}
   \|u_\tau^k\|^2_{\dot H^{1+s}(\R^d)} \leq  \frac{\HH(u_\tau^{k-1})-\HH(u_\tau^k)}{\tau},  \qquad k=1,2,\ldots.
   \end{equation}
\end{lemma}

\begin{proof}
Since $u_\tau^k\in D(\FF_s)\subset L^2(\R^d)$ and $(u\log u)_+\leq u^2$, then $u_\tau^k\in D(\HH)$ for any $k\geq 0$.


Let us fix $k\geq 1$. For $t\geq 0$, we denote by $S_t$ the $0$-flow generated by the entropy $\HH$, and we define $w_t:=S_t(u_\tau^k)$. By Proposition \ref{prop:EntropyFlow}, 
$S_t$ coincides with the heat semigroup on $\R^d$.
By uniqueness of the solution of the Cauchy problem for the heat equation, we have the representation
\begin{equation}\label{eq:HeqC}
w_t=\Gamma_t*u_\tau^k, \qquad \Gamma_t(x):=\frac{1}{(2\pi t)^{d/2}}e^{-|x|^2/(4t)},
\end{equation}
where $*$ denotes the convolution with respect to the space variable $x$.
For the relation with convolution and Fourier transform, by \eqref{eq:HeqC} we have 
\begin{equation}\label{eq:HeqF}
\hat w_t(\xi)=\hat\Gamma_t(\xi)\hat u_\tau^k(\xi).
\end{equation}
We also recall that the Fourier transform of $\Gamma_t$ has the expression
 \begin{equation}\label{FGamma}
\hat\Gamma_t(\xi)=e^{-t|\xi|^2}.
\end{equation}
The Cauchy problem for the heat equation 
in the Fourier setting can be written as a family depending on $\xi\in\R^d$ of Cauchy problems 
\begin{equation}\label{HeqF}
\begin{cases}
\de_t \hat w_t(\xi)= -|\xi|^2\hat w_t(\xi) & t\in(0,+\infty), \\
\lim_{t\to 0}\hat w_t(\xi)=\hat u_\tau^k(\xi). &
\end{cases}
\end{equation}
It is easy to prove that 
$w_t\in \dot H^{1+s}(\mathbb{R}^d)$ for any $t>0$.
Indeed, by \eqref{eq:HeqF} we have
\[\begin{aligned}
 \|w_t\|^2_{\dot H^{1+s}(\R^d)} &= (2\pi)^{-d}\int_{\R^d}|\xi|^{2(1+s)}|\hat w_t(\xi)|^2\,\d\xi 
 =(2\pi)^{-d}\int_{\mathbb{R}^d}|\xi|^{2(1+s)}|\hat\Gamma_t(\xi)|^2|\hat u_\tau^k(\xi)|^2\,\d\xi\\
 &\le C_t\|u_\tau^k\|^2_{\dot H^{s}(\mathbb{R}^d)} = 2C_t\FF_s(u_\tau^k)<+\infty,
\end{aligned}\]
where, using \eqref{FGamma}, 
\begin{equation}\label{Ct}
C_t:=\max_{\xi\in\mathbb{R}^d}|\xi|^2|\hat\Gamma_t(\xi)|^2=e^{-2}t^{-2}.
\end{equation} 

We define the function $g:[0,+\infty)\to\R$ by
$g(t):=\FF_s(w_t)$. 
We prove that $g$ is differentiable in $(0,+\infty)$,
continuous at $t=0$, and 
\begin{equation}\label{gprime}
g'(t)=
  -\| w_t\|^2_{\dot H^{1+s}(\R^d)}
 \qquad \forall\,t\in(0,+\infty).
\end{equation}
Indeed, taking into account that $ |\hat w_t(\xi)|^2=\hat w_t(\xi)\overline{\hat w_t(\xi)}=\hat w_t(\xi)\hat w_t(-\xi)$, 
by \eqref{HeqF} we have that
$$ \de_t |\hat w_t(\xi)|^2= -2|\xi|^2 |\hat w_t(\xi)|^2 \qquad\forall\,(t,\xi)\in (0,+\infty)\times\R^d.$$
Since for any $\xi\in\R^d$ the function $t\mapsto |\hat w_t(\xi)|^2$ belongs to $C^1(0,+\infty)$ and
$$\Big|\de_t |\xi|^{2s} |\hat w_t(\xi)|^2\Big|= 2 |\xi|^{2s+2} |\hat w_t(\xi)|^2 \leq 2 C_t |\xi|^{2s} |\hat u_\tau^k(\xi)|^2,$$
 we can differentiate under the integral sign obtaining that
\begin{equation*}
\begin{aligned}
    g'(t) &=  \frac{1}{2(2\pi)^d} \frac{\d}{\d t}  \int_{\R^d}|\xi|^{2s}|\hat w_t(\xi)|^2\,\d\xi \\
    &=-\frac{1}{(2\pi)^d} \int_{\R^d}|\xi|^{2s}|\xi|^2|\hat w_t(\xi)|^2 \,\d\xi
    =-\|w_t\|^2_{\dot H^{1+s}(\mathbb{R}^d)}
    \end{aligned}
\end{equation*}
and \eqref{gprime} is proved. 
Since $0< \hat\Gamma_t(\xi)\leq 1$ we have 
$|\hat w_t(\xi)|^2= |\hat\Gamma_t(\xi)\hat{u}_\tau^k(\xi)|^2\leq |\hat {u}_\tau^k(\xi)|^2$ and then
$\FF_s(w_t)\leq \FF_s({u}_\tau^k)$, i.e., $g(t)\leq g(0)$  for any $t\in (0,+\infty)$. 
Since $\FF_s$ is lower semi continuous with respect to the narrow convergence (Proposition \ref{prop:FP}),
we have that $\liminf_{t\to 0^+}g(t)\geq g(0)$ and the continuity of $g$ at $t = 0$ is proved.

Applying Lagrange's mean value Theorem to $g$ in the interval $[0,t]$,
for any $t>0$ there exists $\theta(t)\in (0,t)$ such that, recalling the definition of $g$ and \eqref{gprime},
\[
\frac{\FF_s(u_\tau^k) -\FF_s(S_t(u_\tau^k))}{t} =  \|S_{\theta(t)}(u_\tau^k)\|^2_{\dot H^{1+s}(\R^d)}.
\]
From this equality and the definition \eqref{def:DVF}, 
by the lower semicontinuity of the $\dot H^{1+s}(\R^d)$ semi-norm with respect to the narrow convergence it follows that
\[
 \|u_\tau^k\|^2_{\dot H^{1+s}(\R^d)} \leq   \mathfrak{D}_\HH \FF_s(u_\tau^k).
\]
Finally, by Propositions \ref{prop:FI} and \ref{prop:EntropyFlow}, we obtain the estimate \eqref{regestimate} 
 and $u_\tau^k\in H^{1+s}(\mathbb{R}^d)$.
\end{proof}


Integrating the estimate \eqref{regestimate}  with respect to time, we obtain the following space-time bound on the discrete solution $u_\tau$.

\begin{corollary}\label{cor:HR}
Let $u_0\in  D(\FF_s)$, $\tau>0$, $\{u_\tau^k:k=0,1,2,\ldots\}$ the sequence given by {\rm Proposition \ref{prop:existenceMM}} and
$u_{\tau}$ the corresponding discrete piecewise constant approximate solution defined in \eqref{PC}.
Then
$u_{\tau}(t)\in H^{1+s}(\mathbb{R}^d)$ for every $t>0$ and there exists $C>0$ depending only on the dimension $d$
such that
\begin{equation}\label{RI1}
\int_{0}^T\|u_{\tau}(t)\|_{ \dot H^{1+s}(\mathbb{R}^d)}^2\,\d t \le \HH(u_0) + C\Big(1 +T\FF_s(u_0)+\int_{\R^d}|x|^2 u_0(x) \,\d x\Big)
\end{equation} 
 for any $T>0$.
\end{corollary}
\begin{proof}
Let $T>0$ and $N:= [T/\tau]+1$. 
Using \eqref{regestimate} and the definition of $u_\tau$ we obtain
\begin{equation}\label{RI2}
\int_{0}^T\|u_{\tau}(t)\|_{\dot H^{1+s}(\mathbb{R}^d)}^2\,\d t \leq \sum_{k={1}}^{N}\tau\|u^k_{\tau}\|_{\dot H^{1+s}(\mathbb{R}^d)}^2
\leq  \HH(u_0) -\HH(u_\tau^N).
\end{equation}
Using Jensen's inequality, it is not difficult to prove that (see for instance \cite{BCC08}),
\begin{equation}\label{Carleman}
	\HH(u) \geq -\frac{1}{e} - \frac{d}{2}\log(4\pi)-\frac{1}{4}\int_{\R^d}|x|^2u(x)\,\d x, \qquad \forall\,u\in D(\HH).
\end{equation}
By \eqref{Carleman} and \eqref{boundA}  we obtain
$$ -\HH(u_\tau^N)\leq C\Big(1+ T\FF_s(u_0)+\int_{\R^d}|x|^2 u_0(x)\,\d x \Big)$$
for $C$ depending only on the dimension $d$.
By the last inequality and \eqref{RI2} we have \eqref{RI1}. 
\end{proof}

\subsection{Improved convergence}
Thanks to the estimate of Corollary \ref{cor:HR} we obtain the following result of convergence.
This convergence will be fundamental in order to obtain the weak formulation of the equation in \eqref{equation}.

\begin{lemma}\label{lemma:convergence2}
Let $u_0\in  D(\FF_s)$,  $u_\tau$ the piecewise constant curve defined in \eqref{PC} for any $\tau>0$.
Given a vanishing sequence $\tau_n$, let
$u_{\tau_{n}}$ be a convergent subsequence (not relabeled) given by {\rm Theorem \ref{th:convergence1}} and
$u$ its limit curve.
Then, for any $T>0$ we have
$u\in L^2((0,T); H^{1+s}(\R^d))$ and
\begin{equation}\label{basicconv}
 u_{\tau_{n}}\to u \quad \mbox{strongly in } L^2((0,T);H^{1+r}({\R}^d)) \quad \mbox{ as $n\to\infty$, }  \forall \, r< s.
\end{equation}
\end{lemma}

\begin{proof}
Let $r<s$. By \eqref{basicestimate} and lower semicontinuity we have 
\begin{equation}\label{Bbasic}
\| u_{\tau_n}(t)\|^{2}_{\dot H^{s}(\R^d)}\leq 2\FF_s(u_0), \quad \| u(t)\|^{2}_{\dot H^{s}(\R^d)}\leq 2\FF_s(u_0) 
\quad \forall\,t\in[0,+\infty).
\end{equation}
By \eqref{narrowconv} and Corollary \ref{cor:conv} we obtain
\begin{equation}\label{basicconv2}
\lim_{n\to+\infty}\| u_{\tau_n}(t) - u(t)\|^2_{H^{r}(\R^d)}= 0, \qquad \forall t\in[0,+\infty).
\end{equation}
By Corollary \ref{cor:HR} and lower semicontinuity we have 
\begin{equation}\label{RI12}
\int_{0}^T\|u(t)\|_{ \dot H^{1+s}(\mathbb{R}^d)}^2\,\d t \le \HH(u_0) + C\Big(1 +T\FF_s(u_0)+\int_{\R^d}|x|^2 u_0(x) \,\d x\Big)
\end{equation} 
for any $T>0$.

Using the interpolation \eqref{interpolation}, we can write
$$
 \| u_\tau(t) - u(t)\|_{H^{1+r}(\R^d)} \leq  \| u_\tau(t) - u(t)\|^{1-\theta}_{H^{r}(\R^d)}  \| u_\tau(t) - u(t)\|^{\theta}_{H^{1+s}(\R^d)},
$$
for $\theta=1/(1+s-r)\in(0,1)$ and for a.e. $t\in(0,+\infty)$. 
Fixing $T>0$, by H\"older's inequality we obtain
$$
\begin{aligned}
 & \int_{0}^T  \| u_{\tau_n}(t) - u(t)\|^2_{H^{1+r}(\R^d)}\,\d t  \\
& \le\int_{0}^T \| u_{\tau_n}(t) - u(t)\|^{2(1-\theta)}_{H^{r}(\R^d)} \| u_{\tau_n}(t) - u(t)\|^{2\theta}_{H^{1+s}(\R^d)}\,\d t\\
&\le\Big(\int_{0}^T \| u_{\tau_n}(t) - u(t)\|^2_{H^{r}(\R^d)}\,\d t\Big)^{1-\theta} \Big(\int_{0}^T\| u_{\tau_n}(t) - u(t)\|^{2}_{H^{1+s}(\R^d)}\,\d t\Big)^{\theta}.
\end{aligned}
$$
By estimate \eqref{RI1} and \eqref{RI12} the factor $\int_{0}^T\| u_{\tau_n}(t) - u(t)\|^{2}_{H^{1+s}(\R^d)}\,\d t$ is bounded.
Finally, by the previous inequality, \eqref{basicconv2} and \eqref{Bbasic} 
we obtain \eqref{basicconv} using dominated convergence.

\end{proof}

\subsection{Weak formulation  the equation for the discrete solution}

In order to obtain a sort of  weak formulation of the equation for the discrete solution, 
we use the flow interchange estimate with the $(-\lambda)$-flow generated by
a potential energy as in Proposition \ref{prop:PotentialFlow}.
Preliminarily we compute the derivative of the energy functional $\FF_s$ along the flow of a smooth vector field.

\begin{lemma}\label{lemma:DEL}
Let $\eta \in C^\infty_c(\R^d;\R^d)$ and $X_t:\R^d\to\R^d$, $t\in\R$, be the flow associated to $\eta$
defined, for any $x\in\R^d$ as the unique global solution of the problem
\begin{equation}\label{flowRd}
\begin{cases}
\frac{\d}{\d t} X_t(x)=\eta(X_t(x)), & t\in\R\\
X_0(x)=x.
\end{cases}
\end{equation}
Let $u\in H^{1+s}(\R^d)\cap\PP_2(\R^d)$ and $u_t:=(X_t)_\# u$.
Then the map $t\mapsto \FF_s(u_t)$ is differentiable at $t=0$ and 
\begin{equation}\label{eq:diffFeta}
\frac{\d}{\d t}  \FF_s(u_t)_{|t=0}=-N(u,\eta),
\end{equation}
where $N:H^{1+s}(\R^d)\times  C^\infty_c(\R^d;\R^d)\to \R$ is defined in \eqref{N11}.
\end{lemma}
\begin{proof}
Since $\eta \in C^\infty_c(\R^d;\R^d)$, then 
for any $t\in\R$, the map $X_t$ is a $C^\infty$ diffeomorphism of $\R^d$ and $X_t^{-1}=X_{-t}$.
 Moreover if $x\not\in\supp\,\eta$, then $X_t(x)=x$.
Since 
$$
\begin{cases}
\frac{\d}{\d t} \nabla X_t=\nabla \eta(X_t) \nabla X_t , & t\in\R\\
\nabla X_0=I,
\end{cases}
$$
where we used the notation $\nabla\eta$ and $\nabla X_t$ for the Jacobian matrices of $\eta$ and $X_t$,
there exists a constant $L>0$ such that
$$|X_t(x)-X_t(y)|\leq L |x-y|, \qquad \forall\,x,y\in\R^d,\quad \forall\,t\in[-1,1].$$
Recalling the formula \eqref{densityPF}, $u_t(x)= u(X_{-t}(x))\det(\nabla X_{-t}(x))$.
Observing that the map $x\mapsto \det(\nabla X_{-t}(x))$ belongs to $C^\infty(\R^d;\R)$ 
and  $\det(\nabla X_{-t}(x))=1$ for any $x\not\in \supp\,\eta$,
there exists a constant $C>0$ such that
\begin{equation}\label{eq:BE}
 \|u_t\|_{H^{1+s}(\R^d)}\leq C  \|u\|_{H^{1+s}(\R^d)}, \qquad \forall t\in [-1,1].
 \end{equation}
See, for instance, \cite[Corollary 1.60 and Theorem 1.62]{BCD}.

Using the formula 
$|a|^2-|b|^2=(\bar a+\bar b)(a-b)+\bar ab-\bar ba$ valid for $a,b\in\C$, 
by \eqref{Fconj} we have
\begin{equation}\label{Fdiff}
\FF_s(u_t)-\FF_s(u) =
\frac12 \frac1{(2\pi)^d}\int_{\R^d}|\xi|^{2s}\big(\hat u_t(-\xi)+\hat u(-\xi)\big)\big(\hat u_t(\xi)-\hat u(\xi)\big)\,\d\xi,
\end{equation}
because
$$ \int_{\R^d}|\xi|^{2s}\hat u_t(-\xi)\hat u(\xi)\,\d\xi
= \int_{\R^d}|\xi|^{2s}\hat u_t(\xi)\hat u(-\xi)\,\d\xi .$$

Let $m=[s/2]$. 
If $s\in[2m,2m+1]$, by \eqref{Fdiff}, using the Plancherel identity \eqref{Plancherel} and the definition \eqref{defLs}, 
we obtain
\begin{equation}\label{elv1}
 \begin{aligned}  \frac{\FF_s(u_t)-\FF_s(u)}{t} &=
\frac12 \frac1{(2\pi)^d} \int_{\R^d}|\xi|^{2(s-m)}\big(\hat u_t(-\xi)+\hat u(-\xi)\big)\frac{|\xi|^{2m}(\hat u_t(\xi)-\hat u(\xi))}{t}\,\d\xi \\
&= \frac12 \int_{\R^d} \LL_{s-m}(u_t + u)  \LL_m\left(\frac{u_t-u}{t}\right)\,\d x .
\end{aligned}
\end{equation}
Analogously, if $s\in(2m+1,2m+2)$, we write
\begin{equation}\label{elv2}
 \begin{aligned}   \frac{\FF_s(u_t)-\FF_s(u)}{t} &=
\frac12 \frac1{(2\pi)^d} \int_{\R^d}\xi|\xi|^{2(s-m-1)}\big(\hat u_\delta(-\xi)+\hat u(-\xi)\big)\cdot\xi\frac{|\xi|^{2m}(\hat u_\delta(\xi)-\hat u(\xi))}{\delta}\,\d\xi\\
& = \frac12 \int_{\R^d} \nabla\LL_{s-m-1}(u_t + u) \cdot\nabla \LL_m\left(\frac{u_t-u}{t}\right)\,\d x .
\end{aligned}
\end{equation}
Moreover $u_t \to u$ narrowly as $t\to 0$. Indeed, for $\varphi:\R^d\to\R$ continuous and bounded, 
by the definition of $(X_t)_\#u$ and dominated convergence theorem we have that
$\int_{\R^d}\varphi(x)u_t(x)\,\d x= \int_{\R^d}\varphi(X_t(x))u(x)\,\d x \to \int_{\R^d}\varphi(x)u(x)\,\d x$ as $t\to 0$.

Thanks to \eqref{eq:BE} and the narrow convergence of $u_t$ to $u$ we can apply Lemma \ref{lemma:conv}
obtaining that
\begin{equation}\label{cC1}\begin{split}
 &\|\LL_{s-m}u_t-\LL_{s-m}u\|_{L^2(\R^d)}\to 0 \qquad \mbox{if } s\in[2m,2m+1),\\
&\LL_{s-m}u_t\to \LL_{s-m}u \qquad \mbox{weakly in } L^2(\R^d) \quad \mbox{if } s=2m+1,\\
&\|\LL_{s-m-1}u_t-\LL_{s-m-1}u\|_{H^1(\R^d)}\to 0 \qquad \mbox{if } s\in(2m+1,2m+2),
\end{split}\end{equation}
as $t\to 0$.

For every $\xi\in\R^d$
we define $g_\xi:\R\to\R$ by $g_\xi(t):= \hat u_t(\xi)$.
We prove that $g_\xi\in C^1(\R)$
and 
\begin{equation}\label{gder}
 g'_\xi(t)= -\widehat{\div(\eta u_t)}(\xi).
\end{equation}
Indeed, by definition of image measure, we have
$$ g_\xi(t)=\hat u_t(\xi)=\int_{\R^d}e^{-i\xi\cdot X_t(x)}u(x)\,\d x.$$
Using this expression, by dominated convergence Theorem, we have that
$$
\begin{aligned}
\frac{g_\xi(t+h)-g_\xi(t)}{h} &=\int_{\R^d}\frac1h(e^{-i\xi\cdot X_{t+h}(x)}- e^{-i\xi\cdot X_t(x)}) u(x)\,\d x
 \\
&\to \int_{\R^d} e^{-i\xi\cdot X_t(x)}(-i\xi\cdot\eta(X_t(x)))u(x)\,\d x 
\end{aligned}
$$
as $h\to 0$.
Moreover, taking into account  the definition of image measure and \eqref{Frule},
\begin{equation}\label{divetau}
\int_{\R^d} e^{-i\xi\cdot X_t(x)}(-i\xi\cdot\eta(X_t(x)))u(x)\,\d x = \int_{\R^d} e^{-i\xi\cdot x}(-i\xi\cdot\eta(x))u_t(x)\,\d x = -\widehat{\div(\eta u_t)}(\xi).
\end{equation}
The continuity of $g'_\xi$ follows from the expression above and the regularity of the maps $t\mapsto X_t(x)$,
using dominated convergence Theorem.

Using fundamental theorem of calculus and Jensen's inequality, we have
\begin{equation}\label{eq:BEaus}
\begin{split}
\left\|\frac{u_{t}-  u}{t} \right\|^2_{H^s(\R^d)} &=\int_{\R^d}(1+|\xi|^2)^s\left| \frac{\hat u_t(\xi)-\hat u(\xi)}{t}\right|^2\,\d\xi \\
& = \int_{\R^d}(1+|\xi|^2)^s\left| \frac{g_\xi(t)-g_\xi(0)}{t}\right|^2\,\d\xi \\
&= \int_{\R^d}(1+|\xi|^2)^s\left| \frac{1}{t}\int_0^t g'_\xi(r)\,\d r \right|^2\,\d\xi \\
&\leq  \int_{\R^d}(1+|\xi|^2)^s\frac{1}{t}\int_0^t \left| g'_\xi(r)\right|^2\,\d r \,\d\xi \\
&=  \frac{1}{t}\int_0^t \int_{\R^d}(1+|\xi|^2)^s\left| g'_\xi(r)\right|^2 \,\d\xi \,\d r\\
&=  \frac{1}{t}\int_0^t \int_{\R^d}(1+|\xi|^2)^s\left|\widehat{\div(\eta u_r)}(\xi)\right|^2 \,\d\xi \,\d r\\
&=  \frac{1}{t}\int_0^t  \| \div(\eta u_r)\|^2_{H^{s}(\R^d)} \d r.
\end{split}
\end{equation}
From the estimate \eqref{eq:BE} it follows that there exists $C>0$, depending on $\eta$, such that
\begin{equation}\label{eq:BE1}
  \|\div(\eta u_r)\|_{H^{s}(\R^d)} \leq \tilde C \|u_r\|_{H^{1+s}(\R^d)}\leq C  \|u\|_{H^{1+s}(\R^d)}, \qquad \forall r\in [-1,1].
\end{equation}
By \eqref{eq:BEaus} and \eqref{eq:BE1} it follows that 
\begin{equation}\label{Diffestim}
\left\|\frac{u_{t}-  u}{t} \right\|_{H^s(\R^d)}\leq C \|u\|_{H^{1+s}(\R^d)}, \qquad \forall t\in [-1,1], \; t\neq 0.
\end{equation}
Moreover, by Lagrange mean value, \eqref{gder} and \eqref{divetau} we obtain
\begin{equation}\label{Diffestim2}
\left|\frac{\hat u_{t}(\xi)-\hat u(\xi)}{t} \right|  \leq  |\xi|\|\eta\|_\infty \qquad \forall \xi\in\R^d,\; \forall t\in [-1,1], \; t\neq 0.
\end{equation}

Since, by \eqref{gder}, $\lim_{t\to 0} \frac{\hat u_{t}(\xi)- \hat u(\xi)}{t} = -\widehat{\div(\eta u)}(\xi)$ for any $\xi\in\R^d$, 
and \eqref{Diffestim}  \eqref{Diffestim2} hold,
we can apply Lemma \ref{lemma:conv} and we obtain
\begin{equation}\label{cC2}\begin{split}
&\LL_{m}\Big(\frac{u_{t}-  u}{t}\Big)\to \LL_{m}(-\div(\eta u)) \qquad \mbox{weakly in } L^2(\R^d) \quad \mbox{if } s=2m,\\
 &\left\|\LL_{m}\Big(\frac{u_{t}-  u}{t}\Big)- \LL_{m} (-\div(\eta u))\right\|_{L^2(\R^d)}\to 0 \qquad \mbox{if } s\in(2m,2m+1],\\
&\left\|\nabla\LL_{m}\Big(\frac{u_{t}-  u}{t}\Big)- \nabla\LL_{m} (-\div(\eta u))\right\|_{L^2(\R^d)}\to 0 \qquad \mbox{if } s\in(2m+1,2m+2),
\end{split}\end{equation}
as $t\to 0$.

Finally, using \eqref{cC1} and \eqref{cC2} we pass to the limit in \eqref{elv1} and  \eqref{elv2} and we obtain
\begin{equation}\label{ending1}
	 \lim_{t \to 0}\frac1t \left( \FF_s(u_t)-\FF_s(u)\right) = 
	-\int_{\R^d} (\LL_{s-m} u) \LL_m(\div (\eta \, u))\,\d x 
\end{equation}
if  $s\in[2m,2m+1]$
and
\begin{equation}\label{ending2}
	 \lim_{t \to 0}\frac1t \left( \FF_s(u_t)-\FF_s(u)\right) = 
	-\int_{\R^d} \nabla((\LL_{s-m-1} u)\cdot\nabla(\LL_m(\div (\eta \, u)))\,\d x \qquad 
\end{equation}	
if $s\in(2m+1,2m+2).$

By \eqref{ending1} and \eqref{ending2} we obtain \eqref{eq:diffFeta}.
\end{proof}

The application of the Flow interchange estimate with the flow generated by a potential energy yields the following
Proposition. We observe that the inequality \eqref{eq:DD} is a sort of discrete weak formulation of the equation in \eqref{equation} 
(see \eqref{eq:DD3} and \eqref{eq5}).

\begin{proposition}
Let $u_0\in  D(\FF_s)$, $\tau>0$, $\{u_\tau^k:k=0,1,2,\ldots\}$ the sequence given by {\rm Proposition \ref{prop:existenceMM}}.
Let $\varphi\in C^\infty_c(\R^d)$ and $\lambda \geq \|D^2\varphi\|_\infty$.
Then 
\begin{equation}\label{eq:DD}
\begin{aligned}
	-\frac{\lambda}{2} W^2(u_\tau^n,u_\tau^{n-1}) & \leq \int_{\R^d} \varphi(x) u_\tau^n(x)\,\d x -  \int_{\R^d} \varphi(x) u_\tau^{n-1}(x)\,\d x
	-\tau N(u_\tau^n,\nabla\varphi) \\
	&\leq \frac{\lambda}{2} W^2(u_\tau^n,u_\tau^{n-1}), \qquad \forall\,n\in\N,
	\end{aligned}
\end{equation}
where $N$ is defined in \eqref{N11}.
\end{proposition}

\begin{proof}
Let us define the functional $\VV:\PP_2(\R^d)\to \R$ by 
$$\VV(u):=\int_{\R^d}\varphi(x)\,\d u(x).$$
Let $n\in\N$. Since by Lemma \ref{lemma:ed}  $u_\tau^n\in H^{1+s}(\R^d)$,
by Proposition \ref{prop:PotentialFlow}
and Lemma \ref{lemma:DEL} for $\eta=-\nabla\varphi$,
we have
\begin{equation*}
 \mathfrak{D}_\VV \FF_s(u_\tau^n):= \limsup_{t \downarrow 0} \frac{\FF_s(u_\tau^n)-\FF_s(S_t(u_\tau^n))}{t} = -\frac{\d}{\d t}  \FF_s(S_t(u_\tau^n))_{|t=0}
 =N(u_\tau^n,-\nabla\varphi),
\end{equation*}
Applying Proposition \ref{prop:FI} to $\VV$  and observing that $N(u_\tau^n,-\nabla\varphi)=-N(u_\tau^n,\nabla\varphi)$,
we obtain
\begin{equation}\label{eq:DD1}
\begin{aligned}
	-\frac{\lambda}{2} W^2(u_\tau^n,u_\tau^{n-1}) - \tau N(u_\tau^n,\nabla\varphi)
	 \leq \VV(u_\tau^{n-1}) - \VV(u_\tau^n)   .
		\end{aligned}
\end{equation}
Analogously, applying Proposition \ref{prop:FI} to $-\VV$ instead of $\VV$ and observing that $-\VV$ still generates 
a $-\lambda$-flow we obtain
\begin{equation}\label{eq:DD2}
\begin{aligned}
	-\frac{\lambda}{2} W^2(u_\tau^n,u_\tau^{n-1}) + \tau N(u_\tau^n,\nabla\varphi)
	 \leq -\VV(u_\tau^{n-1}) + \VV(u_\tau^n)  .
		\end{aligned}
\end{equation}
Finally, the inequality \eqref{eq:DD} follows by \eqref{eq:DD1} and \eqref{eq:DD2}. 
\end{proof}

\subsection{Solution of the problem}

In this Section we prove that the limit curve given by Theorem \ref{th:convergence1} is 
a weak solution of problem \eqref{equation} and we conclude the proof of Theorem \ref{th:main}.

\begin{theorem}
\label{th:weak_form}
If $u\in AC^2([0,+\infty);(\PP_2({\R}^d),W))$ is a limit curve given by {\rm Theorem \ref{th:convergence1}}, 
then $u$ is a solution of the equation in \eqref{equation} in the following  weak form:
for any $\varphi\in C^\infty_c((0,+\infty)\times\R^d)$
\begin{equation}\label{wf}
\int_0^{+\infty}\int_{\R^d} \partial_t\varphi u \,\d x\,\d t+
\int_0^{+\infty}N(u(t),\nabla\varphi(t,\cdot)) \, \d t=0, 
\end{equation}
where $N$ is defined in \eqref{N11}.

\end{theorem}

\begin{proof}
Let $\varphi\in C^\infty_c((0,+\infty)\times\R^d)$, $T>0$ such that $\varphi(t,\cdot)=0$ for any $t>T$.
Let $\lambda \geq \max_{t\in[0,T]}\|D^2\varphi(t,\cdot)\|_\infty$.

Using the notation $u_\tau(t,x):=u_\tau(t)(x)$ and the convention $u_\tau(t):=u_0$ if $t<0$, 
the inequality \eqref{eq:DD} can be rewritten as
\begin{equation}\label{eq:DDbis}
\begin{aligned}
	&-\frac{\lambda}{2} W^2(u_\tau(t),u_\tau(t-\tau)) \\
	& \leq \int_{\R^d} \varphi(t,x) (u_\tau(t,x) - u_\tau(t-\tau,x))\,\d x
	-\tau N(u_\tau(t),\nabla\varphi(t,\cdot)) \\
	&\leq \frac{\lambda}{2} W^2(u_\tau(t),u_\tau(t-\tau)), \qquad \forall \,t\in[0,+\infty), \quad \forall \,\tau>0.
	\end{aligned}
\end{equation}
Dividing the inequality in \eqref{eq:DDbis} by $\tau>0$ and integrating in time, we obtain
\begin{equation}\label{eq:DD3}
\begin{aligned}
	&\left| \int_0^T \int_{\R^d} \frac{\varphi(t,x)-\varphi(t+\tau,x)}{\tau} u_\tau(t,x) \,\d x\,\d t
	-\int_0^T N(u_\tau(t),\nabla\varphi(t,\cdot))\,\d t \right| \\
	&\leq \frac{\lambda}{2\tau} \int_0^T W^2(u_\tau(t),u_\tau(t-\tau))\,\d t.
	\end{aligned}
\end{equation}
We observe that the inequality \eqref{eq:DD3} is a discrete weak formulation of the equation \eqref{equation}.

Let $\tau_n$ be a vanishing sequence given by Theorem \ref{th:convergence1}.

First of all we show that
\begin{equation}\label{eq5}
	\lim_{n\to+\infty}\frac{\lambda}{2\tau_n} \int_0^T W^2(u_{\tau_n}(t),u_{\tau_n}(t-\tau_n))\,\d t = 0.
\end{equation}
Indeed, by \eqref{basicestimate}
\begin{equation*}\label{eq:?}
\begin{aligned}
	&\frac{1}{2\tau_n} \int_0^T W^2(u_{\tau_n}(t),u_{\tau_n}(t-\tau_n))\,\d t \leq	
	\frac{1}{2}\sum_{k=1}^{[T/\tau_n]+1} W^2(u^k_{\tau_n},u_{\tau_n}^{k-1}) \leq \tau_n \FF_s(u_0)
			\end{aligned}
\end{equation*}
and \eqref{eq5} follows.

We pass to the limit in the other two terms in \eqref{eq:DD3}.
By the convergence \eqref{basicconv} and the regualrity of $\varphi$ it follows that
\begin{equation}\label{eq3}
	\lim_{n\to+\infty} \int_0^T \int_{\R^d} \frac{\varphi(t,x)-\varphi(t+\tau_n)}{\tau_n} u_{\tau_n}(t,x) \,\d x\,\d t = -\int_0^T \int_{\R^d} \de_t\varphi(t,x) u(t,x) \,\d x\,\d t.
\end{equation}
Let $m=[s/2]$.
For $s\in[2m,2m+1]$, 
by definition of $N$,
\begin{equation}\label{euler3}
	\int_0^T N(u_{\tau_n}(t),\nabla\varphi(t,\cdot))\,\d t
	= \int_0^{T}\int_{\R^d} \LL_{s-m} (u_{\tau_n}) \LL_m(\div (\nabla\varphi \, u_{\tau_n}))\,\d x\,\d t.
\end{equation}
We observe that
$\|\LL_{s-m}(u_{\tau_n}-u)\|_{L^2(\R^d)}=\|u_{\tau_n}-u\|_{\dot H^{2s-2m}(\R^d)}$.
Let $s\in[2m,2m+1)$, defining $r$ such that $1+r=2s-2m$, it holds $r<s$.
By Lemma \ref{lemma:convergence2}, we have
$\LL_{s-m}u_{\tau_n}\to\LL_{s-m}u$ strongly in $L^2((0,T);L^2(\R^d))$.
If $s=2m+1$ we have $\LL_{m+1}u_{\tau_n}\to\LL_{m+1}u$ weakly in $L^2((0,T);L^2(\R^d))$.
By Lemma \ref{lemma:convergence2}, we have also that $\div (\nabla\varphi \, u_{\tau_n})\to \div (\nabla\varphi \, u)$
strongly in $L^2((0,T);H^r({\R}^d))$ for any $r<s$.
$\LL_m\div (\nabla\varphi \, u_{\tau_n})\to \LL_m\div (\nabla\varphi \, u)$ strongly in $L^2((0,T);H^{r-2m}({\R}^d))$.
In particular, for $r=2m$ we obtain  
$\LL_m\div (\nabla\varphi \, u_{\tau_n})\to \LL_m\div (\nabla\varphi \, u)$ strongly in $L^2((0,T);L^{2}({\R}^d))$.
The convergences above and \eqref{euler3} show that
\begin{equation}\label{eq4}
	\lim_{n\to+\infty} \int_0^T N(u_{\tau_n}(t),\nabla\varphi(t,\cdot))\,\d t = \int_0^T N(u(t),\nabla\varphi(t,\cdot))\,\d t
\end{equation}
when $s\in[2m,2m+1)$.
Analogously we can prove \eqref{eq4} also in the case $s\in(2m+1,2m+2)$.

The proof of \eqref{wf} follows by \eqref{eq:DD3}, \eqref{eq3}, \eqref{eq4} and \eqref{eq5}. 
\end{proof}

We conclude this Section with the proof of Theorem \ref{th:main}. 
The part i) is exactly Proposition \ref{prop:existenceMM}.
The part ii) follows by Theorem \ref{th:convergence1}, the inequality \eqref{Bbasic}, Corollary \ref{cor:conv} and Lemma \ref{lemma:convergence2}.
The part iii) is exactly Theorem \ref{th:weak_form}.

\subsection*{Acknowledgements} 
The author acknowledge  support from the project MIUR - PRIN 2017TEXA3H ``Gradient flows, Optimal Transport and Metric Measure Structure''.
The author is member of the GNAMPA group of the Istituto Nazionale di Alta Matematica (INdAM)
and this research was partially supported by the GNAMPA Project 2019 ``Trasporto ottimo per dinamiche con interazione''.
%


\end{document}